\documentclass[12pt]{amsart}
\usepackage{amsfonts,amssymb,latexsym,amsmath, amsxtra, stmaryrd}
\usepackage[all]{xy}
\usepackage{graphicx}
\usepackage{color}
\theoremstyle{plain}
\newtheorem{theorem}{Theorem}[section]

\newtheorem{lemma}[theorem]{Lemma}
\newtheorem{proposition}[theorem]{Proposition}

\newtheorem*{conjecture*}{Conjecture}
\newtheorem*{challenge*}{Open Problem}

\theoremstyle{definition}
\newtheorem{definition}[theorem]{Definition}
\newtheorem*{definition*}{Definition}

\newtheorem{example}[theorem]{Example}
\theoremstyle{remark}
\newtheorem*{remark}{Remark}

\numberwithin{equation}{section}

\newcommand{\R}{\mathbb R}

\newcommand{\Z}{\mathbb Z}
\newcommand{\C}{\mathbb C}
\def\H{\mathbb H}

\newcommand{\Q}{{\mathbb Q}}

\def\SL{\rm SL}

\def\({\left(}
\def\){\right)}

\newcommand{\ol}[1]{\overline{{#1}}}

\newcommand{\abs}[1]{\left|#1\right|}

\newcommand{\bC}{{\mathbb C}}

\newcommand{\bQ}{{\mathbb Q}}
\newcommand{\bR}{{\mathbb R}}

\newcommand{\bZ}{{\mathbb Z}}

\newcommand{\cF}{{\mathcal F}}
\newcommand{\cG}{{\mathcal G}}

\newcommand{\cS}{{\mathcal S}}

\newcommand{\fa}{{\mathfrak a}}

\def\k2{\frac{k}{2}}

\begin{document}

\title[renormalization and quantum modular forms, part II]{
Renormalization and quantum modular forms, Part II: Mock theta functions}
\author{Yingkun Li}
\address{UCLA Mathematics Department, Box 951555, Los Angeles, CA 90095-1555, USA}
\email{yingkun@math.ucla.edu}

\author{Hieu T. Ngo}
\address{Dept. of Mathematics, University of Michigan, Ann Arbor, MI 48109-1043, USA}
\email{hieu@umich.edu}

\author{Robert C. Rhoades}
\address{Department of Mathematics, Stanford University, Stanford, CA.
94305}
\curraddr{Center for Communications Research,
Princeton, NJ 08540}
\email{rhoades@math.stanford.edu, rob.rhoades@gmail.com}


\thispagestyle{empty} \vspace{.5cm}
\begin{abstract}
Sander Zwegers showed that Ramanujan's mock theta functions are
$q$-hypergeometric series, whose $q$-expansion coefficients are half
of the Fourier coefficients of a non-holomorphic modular form. George
Andrews, Henri Cohen, Freeman Dyson, and Dean Hickerson found a pair
of $q$-hypergeometric series each of which contains half of the
Fourier coefficients of Maass waveform of eigenvalue $1/4$.

This series of papers shows that a $q$-series construction, called
``renormalization'', yields the other half of the Fourier coefficients
from a series which contains half of them.  This construction unifies
examples associated with mock theta functions and examples associated
with Maass waveforms.  Thus confirming a conviction of Freeman Dyson.
This construction is natural in the context of Don Zagier's quantum
modular forms. Detailed discussion of the role quantum modular forms
play in this construction is given.

New examples associated to Maass waveforms are given in Part I.  Part
II contains new examples associated with mock theta functions, and
classical modular forms.  Part II contains an extensive survey of the
``renormalization'' construction.  A large number of examples and open
questions which share similarities to the main examples, but remain
mysterious, are given.

\end{abstract}

\maketitle


\section{Introduction}\label{sec:intro}\label{sec:Intro}
Two $q$-hypergeometric series which lie at the interface with algebraic number theory are  
\begin{align}
\sigma(q) :=& \sum_{n= 0}^\infty S(n) q^n =  \sum_{n=0}^\infty \frac{q^{\frac{1}{2}n(n+1)}}{(-q;q)_n}
= 1 + q - q^2 + \cdots + 2q^{55} + q^{57}  + \cdots + 6 q^{1609} + \cdots 
\label{eqn:sigma}  
\\
\sigma^*(q):= &  \sum_{n=0}^\infty S^*(n) q^n = 2 \sum_{n=1}^\infty \frac{(-1)^n q^{n^2}}{(q;q^2)_n}
= -2q -2q^{2} + \cdots + 2q^{66} - 2q^{67} - 4q^{70} + \cdots 
\label{eqn:sigma*}
\end{align}
where $(x;q):= \prod_{j=0}^{n-1} (1-xq^j)$ and $(x;q)_\infty:= \prod_{j=0}^\infty (1-xq^j)$
and the expansion is valid for $\abs{q} < 1$. 
The first of these series was investigated by Ramanujan \cite{andrewsEuler}. 
 The connection to algebraic number theory was discovered by Andrews, Dyson,  and Hickerson 
 \cite{ADH} and Cohen \cite{cohen}. 
Andrews, Dyson, and Hickerson 
showed that 
$$S(n) = T(24n+1)  \ \ \ \text{ and } \ \ \ S^*(n) = T(1-24n)$$
where $T(n)$ are the coefficients of a Maass waveform 
$$\varphi_0(z):= 
 \sqrt{y} 
\sum_{n\in 24\Z+1} T(n) K_0(2\pi \abs{n} y/24) e^{2\pi i n x/24},$$
which satisfies the modular properties
$$\varphi_0(-1/2z) = \ol{\varphi_0(z)} \ \ \ \text{ and } \ \ \ \varphi_0(z+1) = e^{2\pi  i/24} \varphi_0(z),$$ 
where $K_0$ is a $K$-Bessel function and $z := x+iy$ with $x, y\in \R$ such that  $y>0$.
Thus, the positive Fourier coefficients of the non-holomorphic modular form $\varphi_0(z)$ are encoded by 
$\sigma(q)$, whereas the negative Fourier coefficients are encoded by $\sigma^*(q)$.

Dyson \cite{dysonGarden} wrote the following:
\begin{quote}\emph{
This pair of functions $\sigma(q)$ and $\sigma^*(q)$ is today an isolated curiosity. 
But I am convinced that, like so many other beautiful things in Ramanujan's garden, it will
turn out to be a special case of a broader mathematical structure. 
There probably
exist other sets of two or more functions with coefficients related by cross multiplicativity,
satisfying identities similar to those which Ramanujan discovered
for his $\sigma(q)$. I have a hunch that such sets of cross-multiplicative functions will
form a structure within which the mock theta-functions will also find a place. But
this hunch is not backed up by any solid evidence. 
}
\end{quote}

This paper gives a construction which explains how the
$q$-hypergeometric series defining mock theta functions reveal all of
the Fourier coefficients necessary to complete the (non-holomorphic)
modular form. The construction requires a ``renormalized'' version of
the $q$-series and the theory of ``quantum modular forms''.  The
renormalization identities are the types of identities referred to by
Dyson in the above quote.  In Part I of this series of papers
renormalization was used to produce the other half of the Fourier
coefficients of a Maass wave form from a series which contained half
of them.  Thus, these works confirm Dyson's conviction.

Instead of completing the modular object by passing through
non-holomorphic modular forms (see Section \ref{sec:mocktheta} for
details), we construct a pair of holomorphic functions.  One of the
two functions is defined on the Poincar\'e upper half-plane $\H$,
while the other is defined on $\H^- = \{ x+iy: x, y\in \R, y <
0\}$. Together the pair of functions encode all of the Fourier
coefficients. Moreover, the series agree on a subset of $S \subset
\Q$.  This subset of points lets the function ``leak'' from $\H$ to
$\H^-$.  The function on $S$ is a quantum modular form. A new feature
of these works is the relationship between the subset $S$ and the
cuspidality of an associated modular form (the Zagier shadow of the
mock theta function, see Section \ref{sec:mocktheta}).

Section \ref{sec:Results} contains new examples concerning weight $3/2$ mock theta
functions and holomorphic modular forms.  An extensive survey of the
renormalization construction is given in Section \ref{sec:Recurrence}.
A number of open problems and remaining mysterious are discussed in
Section \ref{sec:Other}. These sections should be valuable for future
investigation into these ideas. The remainder of this section gives
definitions and examples of mock theta functions, quantum modular
forms, and ``renormalization''.

\subsection{Mock Theta Functions}\label{sec:mocktheta}
As a result of Zwegers's thesis \cite{zwegers}, we understand that
Ramanujan's mock theta functions contain ``half'' of the Fourier
coefficients of a non-holomorphic modular form.  To be precise, each
of Ramanujan's seventeen examples are instances of Zagier's
\cite{zagierBourbaki, DMZ} definition (see also Ono's survey
\cite{Ono3}).
\begin{definition}
A \emph{mock theta function} is a $q$-series $H(q) = \sum_{n=0}^\infty a_n q^n$ such that
 there exists a rational number $\lambda \in \Q$ and a unary theta function $g(z) = 
 \sum_{n\in \Q^+} b_n q^n$ of weight $k$, 
 such that 
 $$\varphi_H(z) = q^{\lambda} H(q) - g^*(z),$$ with $q := e^{2\pi i z}$, 
 is a non-holomorphic  modular form of weight $2-k$, where 
$$g^*(z) = (4\pi)^{1-k} \sum_{n\in \Q^+} n^{2-k} \overline{b_n} \Gamma\( 2-k, 4\pi ny\) q^{-n}$$
with $\Gamma(w, t) = \int_t^\infty u^{w-1} e^{-u} du$, the incomplete Gamma function.  
The modular form $g$ is called the \emph{Zagier shadow} of $f$.
\end{definition}


For example, the Ramanujan's third order mock theta function is
\begin{equation}\label{eqn:f_def}
f(q):= \sum_{n=0}^\infty \frac{q^{n^2}}{(-q;q)_n^2}
= 1 + q - 2q^2 + 3q^3 - 3q^4 + 3q^5 -5q^6 + 7q^7 - 6q^8 + \cdots -18520 q^{100} + \cdots 
\end{equation}
Its Zagier shadow is the 
the weight $3/2$ unary theta function 
\begin{equation}\label{eqn:f_shadow}
g(z) =  \frac{1}{\sqrt{6}} 
\sum_{n\equiv 1\pmod{6}} n q^{\frac{n^2}{24}} = \frac{1}{\sqrt{6}} \sum_{n=1}^\infty \( \frac{-12}{n}\) n q^{\frac{n^2}{24}}.
\end{equation}
More precisely, the function  
$$\phi_f(z) = q^{-\frac{1}{24}} f(q)  - g^*(z), $$
with 
  \begin{equation}\label{eqn:R3}
g^*(z) =   \frac{2}{\sqrt{\pi}} \sum_{n=1}^\infty \(\frac{-12}{n}\)\Gamma\(\frac{1}{2}, \frac{\pi n^2 y}{6}\) q^{\frac{-n^2}{24}}
\end{equation}
is a weight $\frac{1}{2}$ non-holomorphic modular form with respect to $\Gamma(2)$ (see  \cite{zagierBourbaki} page 07).
In particular it satisfies
$$ \frac{1}{\sqrt{2z+1}} \phi_f\( \frac{z}{2z+1}\) =  e^{\frac{2\pi i }{12}}  \phi_f(z) \ \ \ \text{ and } \ \ \ 
\phi_f( z+ 1) = e^{-\frac{2\pi i }{24}} \phi_f(z).$$

\subsection{Quantum Modular Forms}\label{subsec:quantum}
One of the most famous modular forms is the Dedekind $\eta$-function 
$$\eta(z) := q^{\frac{1}{24}} \prod_{n=1}^\infty (1-q^n) \ \ \text{ where } \ \ q= e^{2\pi i z}$$
which satisfies the modular identities  
$$\eta(z+1 ) = e^{\frac{\pi i }{12}} \eta(z) \ \ \ \text{ and } \ \ \ \eta\( -\frac{1}{z}\) = \sqrt{-iz} \eta(z)$$
for all $z = x+iy$ with $x, y\in \R$ and $y>0$.  Thus $\eta(z)$ is a weight $1/2$ modular form. 
Euler's famous pentagonal number theory gives 
$$\eta(z) = \sum_{n>0} \(\frac{12}{n} \) q^{\frac{n^2}{24}} = q^{\frac{1}{24}} \( 1 - q - q^2 + q^5 + q^7 - \cdots \),$$
where $\( \frac{\cdot}{\cdot}\)$ is the Kronecker symbol. 

In a 1997 lecture at Max Planck, Maxim Kontsevich defined  
\begin{equation}
F(q) := 1 + (1-q) + (1-q)(1-q^2) + \cdots = \sum_{n=0}^\infty (q;q)_n
\end{equation}
Kontsevich observed that while the series diverges for both $\abs{q} < 1$ and $\abs{q}>1$, and also for most points on the unit circle, it terminates (and hence converges!) at all roots of unity. 

Zagier proved the ``strange identity''  \begin{equation}\label{eqn:identity}
q^{\frac{1}{24}} F(q) = -\frac{1}{2} \sum_{n>0} n \(\frac{12}{n}\) q^{\frac{n^2}{24}} 
 = -\frac{1}{2} q^{\frac{1}{24}}\( 1 - 5q - 7q^2 + 11 q^5 + 13q^{7} + \cdots \) =: -\frac{1}{2} \widetilde{\eta}(z)
 \end{equation}
 where the identity holds as an asymptotic expansion at each root of unity.  The function 
 $\widetilde{\eta}(z)$ is an ``Eichler integral'' or ``half-derivative'' of the modular form $\eta(z)$. 

Define $\varphi(\alpha):= e^{2\pi i \alpha} F(e^{2\pi i \alpha})$ for $\alpha \in \Q$. Zagier, using \eqref{eqn:identity},
showed that $\varphi(\alpha)$ satisfies the almost-modular identities
$$\varphi(\alpha+1) = e^{\frac{\pi i }{12}} \varphi(\alpha) \ \ \text{ and } \ \ 
\varphi(\alpha) + (i\alpha)^{-\frac{3}{2}} \varphi\(-\frac{1}{\alpha}\) = \frac{\sqrt{3i}}{2\pi} \int_0^\infty (z-\alpha)^{-\frac{3}{2}} \eta(z) dz =: g(\alpha).$$
The function $g:\R \to \C$ is a $C^\infty$ function which is real-analytic everywhere except at $\alpha = 0$ (see  
the theorem of Section 6 of \cite{zagierStrange}).

This example, and a few others from quantum invariants of 3-manifolds, led Zagier to introduce a notion of 
``quantum modular form'' \cite{zagierQuantum}.  Quantum modular forms are functions from $\Q \to \C$ which have 
a nice behavior under the action of $\SL_2(\Z)$ (or a subgroup thereof) on $\Q$.

\begin{definition}
\label{def:quantum}
Let $k\in \frac{1}{2} \Z$ , $S \subset \bQ$ a dense subset, $\Gamma
\subset SL_2(\bZ)$ a subgroup of finite index and $\chi: \Gamma
\longrightarrow \bC^\times$ a finite order character.  A function $f:
S \longrightarrow \bC$ is a \emph{quantum modular form} of weight $k$
and character $\chi$ with respect to $\Gamma$ if for every $\gamma \in
\Gamma$, the period function $h_\gamma:\R \to \C$ given by
\begin{equation}
\label{eq:period_function}
h_\gamma(x): = \chi(\gamma) f(x)  - (f \mid_{k, \gamma}) (x) 
\end{equation}
is $C^\infty$ on $\bR$ and real-analytic at all but a finite set of points.
Here $\mid_{k, \gamma}$ is the weight $k$ slash operator 
$$
(f \mid_{k, \gamma}) (x) = f(\gamma x)(cx+d)^{-k}
$$
when $\gamma=\begin{bmatrix} a & b \\ c & d\end{bmatrix}$.
\end{definition}

\begin{remark}
Recently, many authors have explored the connection between quantum modular forms and ``Ramanujan's radial limits'' associated to his mock theta functions. See the work of Folsom, Ono and the third author \cite{FOR}, as well as the works of Mortenson \cite{mortenson2} and Zudilin \cite{zudilin}.  
\end{remark}

\subsection{Renormalization}\label{sec:RenormalizationIntro}
This section defines the  $q$-series construction called \emph{renormalization}.

\begin{definition}
A \emph{$q$-hypergeometric series} is a sum of the form
$H(q)=\sum_{n=0}^\infty H_n(q)$ where $H_n(q) \in \Q(q)$ and
$H_{n+1}(q)/H_n(q)= R(q,q^n)$ for all $n\ge0$ for some fixed rational
function $R(q, r) \in \Q(q, r)$.
\end{definition}
 Renormalization is a construction
used for extending $H(q)$ to the region $\abs{q}>1$.

\begin{definition}(Renormalization) Let 
$H(q)$ be a mock theta function of weight $k$ defined by the $q$-hypergeometric series 
$H(q)=\sum_{n=0}^\infty H_n(q)$ that converges for $\abs{q}<1$.
Let $g(z) = \sum_{n \in \Q^+} b_n q^n$ be the Zagier shadow of $H$. 
When $\lim_{n\to \infty} H_n(q^{-1})$ converges as a power series in $q$ with $\abs{q}<1$, 
define the  series $\cS[H](q)$ and $\cG[H](q)$ by 
\begin{equation}\label{eqn:sumsoftails}
\cS[H](q) := \sum_{n=0}^\infty \( H_n(q^{-1}) - H_\infty(q^{-1})\) - \cG[H](q)
\end{equation}
where $H_{\infty}(q^{-1}) := \lim_{n\to \infty} H_{n}(q^{-1})$ such that 
\begin{enumerate}
\item  $\cG[H](q)$ vanishes to infinite order at every root of unity $q$ where $H(q^{-1})$ converges. 
\item $\cS[H](q)$ converges for $\abs{q}<1$ and  $ \cS[H](q) 
= C_k \sum_{n \in \Q^{\mp} } \overline{b_n} n^{2-k} q^{\lambda \abs{n}- \alpha}$ 
for some constants $C_k$ and $\alpha$. 
\end{enumerate}
The term $\cS[H](q)$ is called the \emph{shadow} of $H$ and $\cG[H](q)$ is called the \emph{ghost} of $H$. 
\end{definition}

\begin{remark}
In certain cases, the construction in \eqref{eqn:sumsoftails} is known
as the ``sums of tails''.  It appears in Ramanujan's work, as well as
in a number of other papers. See \cite{andrewsEuler, AGL, AJO, BK2,
  zagierStrange}, for example.
\end{remark}

\begin{example}
The mock theta function $f(q) = \sum_{n=0}^\infty f_n(q)$ with $f_n(q)
= \frac{q^{n^2}}{(-q;q)_n^2}$ has Zagier shadow equal to $g(z) =
\frac{1}{\sqrt{6}} \sum_{n=1}^\infty \(\frac{-12}{n}\) n
q^{\frac{n^2}{24}}$.  The following calculations demonstrate that
$\cS[f](q)$ is $-\sum_{n=1}^\infty \(\frac{-12}{n}\)
q^{\frac{n^2-1}{24}}$ and thus gives the Fourier coefficients of
$g^*(z)$, up to the incomplete Gamma function terms.

In the notation above $f_n(q^{-1}) = \frac{q^n}{(-q;q)_n^2}$.  So $f_\infty(q^{-1}) = 0$. Then 
\begin{align*}
\sum_{n=0}^\infty \( f_n(q^{-1}) - f_\infty(q^{-1})\) =& \sum_{n=0}^\infty \frac{q^n}{(-q;q)_n^2}  \\
=& 1 + q - q^2 + 2q^3 - 4q^4 + 5q^5 - 6q^6 + 7 q^7 + \cdots + 392 q^{29} + \cdots.
\end{align*}
Ramanujan's lost notebook (given in \cite{andrewsLost}, valid for
$|q|<1$) contains the identity
$$f(q^{-1}) = \sum_{n=0}^\infty \frac{q^n}{(-q;q)_n^2}  = 2\sum_{n\ge 0} \(\frac{-12}{n}\) q^{\frac{n^2 -1}{24}} +  \frac{1}{(-q)_\infty^{2}}
 \sum_{n\ge 0} (-1)^{n+1} q^{\frac{n(n+1)}{2}}.$$
So $\cS[f](q) =  2\sum_{n\ge 0} \(\frac{-12}{n}\) q^{\frac{n^2 -1}{24}}$ and 
$\cG[f](q) = \frac{1}{(-q)_\infty^2} \sum_{n=0}^\infty (-1)^{n+1} q^{\frac{1}{2} n(n+1)}.$
\end{example}

\begin{remark}
The formal operation of $q\mapsto q^{-1}$ was previously considered by Zwegers \cite{zwegers1} and 
(see Lawerence-Zagier \cite{LZ}), 
Hikami \cite{hikami} and Bringmann, Folsom, and the third author \cite{BFR}.  In all of those cases 
the authors consider cases when $H_{\infty}(q) = 0$.  So renormalization is not needed. 
Recently, Mortenson \cite{mortenson} has given 
a heuristic which given the ``Appell-Lerch'' version of a mock theta function will predict what the 
shadow will be.  This uses a heuristic from joint work with Hickerson \cite{hm} coming from $q\mapsto q^{-1}$. 
This heuristic is closely related to results of the third author with Chern \cite{CR} which relates the Appell-Lerch sums 
to partial theta functions via the Mordell integral. 
\end{remark}

\section{Statement of Results}\label{sec:results}\label{sec:Results}
 In Part I of
these papers renormalization was used to produce the other half of the
Fourier coefficients of a Maass wave form from a series which
contained half of them. 
This section uses renormalization on a pair of mock modular forms to
produce the Fourier coefficients of the Zagier shadow. 
 Let
\begin{align}
M_1(q):=& \sum_{n=0}^\infty M_{1, n}(q) =  \sum_{n=0}^\infty \frac{(-1)^n (q;q)_n}{(-q;q)_n} 
= 1 - \frac{1-q}{2} \sum_{n=0}^\infty \frac{(q;q)_n (-q)^n}{(-q;q)_n}  \\ 
 & = \frac{1}{2} + q - 2q^2 + 4q^3 - 5q^4 + 4q^5 - 4q^6 + 8q^7 - 10q^8 
 + 5q^9 - 4q^{10} + 12q^{11} - \cdots \nonumber \\  
M_2(q):= &\sum_{n=0}^\infty M_{2, n}(q) =  \sum_{n=0}^\infty \frac{q^{n+1} (q^2;q^2)_n}{(q;q^2)_{n+1}}  \\
&= q + 2q^2 + 3q^3 + 3q^4 + 4q^5 + 5q^6 + 4q^7 + 5q^8 + 7q^9 + 5q^{10} + 6q^{11} + 8q^{12} + \cdots \nonumber 
\end{align}
\begin{remark}
The series defining $M_1(q)$ does not technically converge for $\abs{q}<1$.  However, 
standard recurrence relations for the series extend converge into this domain. 
See \eqref{eqn:fine2.4}
and the discussion thereafter. 
\end{remark}

\begin{theorem}\label{thm:mock}
Let $S_0 := \{x \in \Q : =\gamma \cdot x = 0 \text{ for some } \gamma \in \Gamma_0(2)\}$ and 
$S_\infty := \{x \in \Q : =\gamma \cdot x = \infty \text{ for some } \gamma \in \Gamma_0(2)\}$.

\begin{enumerate}
\item[(A)] (Zagier shadow)
$M_1(q)$ is a mock theta function with Zagier shadow proportional 
to $$\frac{\eta(z)^2}{\eta(2z)} = \sum_{n\in \Z} (-1)^n q^{n^2}.$$
Moreover, $\frac{\eta(z)^2}{\eta(2z)}$ is modular with respect to $\Gamma_0(2)$ which vanishes 
at the cusp $0$ (i.e. at all $z\in S_0$) and not at the cusp $\infty$ (i.e. at all $z\in S_\infty$). 

$M_2(q)$ is a mock theta function  with Zagier shadow proportional to 
$$\frac{\eta(2z)^2}{\eta(z)} = \sum_{n\in \Z}  q^{2 \(n+\frac{1}{4}\)^2}$$
Moreover, $\frac{\eta(2z)^2}{\eta(z)}$ is modular with respect to $\Gamma_0(2)$ which 
vanishes at the cusp $\infty$ and not at the cusp $0$.

\item[(B)] (Renormalized shadow)
In the notation above, $$\cS[M_1](q):= 4 \sum_{n=1}^\infty (-1)^n n q^{n^2} \ \ \text{ and } \ \ 
\cG[M_1](q) := 2 \frac{(q;q)_\infty}{(-q;q)_\infty} \sum_{j=1}^\infty \frac{q^j}{1-q^{2j}}$$
since 
$$ 4 \sum_{n=1}^\infty (-1)^n n q^{n^2}= \sum_{n=0}^\infty \(
\frac{(q;q)_n}{(-q;q)_n} - \frac{(q;q)_\infty}{(-q;q)_\infty}\) -2
\frac{(q;q)_\infty}{(-q;q)_\infty} \sum_{j=1}^\infty
\frac{q^j}{1-q^{2j}}. $$ 

In the notation above, $$\cS[M_2](q) :=
\sum_{n=0}^\infty \(n+\frac{1}{4}\) q^{\frac{1}{2}n(n+1)} \ \ \text{
  and } \ \ \cG[M_2](q) := \frac{(q^2;q^2)_\infty}{(q;q^2)_\infty}\(
\frac{1}{4} + \sum_{j=1}^\infty \frac{(-1)^j q^j}{1-q^j}\)$$ since
$$ \sum_{n=0}^\infty \(n+\frac{1}{4}\) q^{\frac{1}{2}n(n+1)}  = - \sum_{n=0}^\infty \( \frac{(q^2;q^2)_n}{(q;q^2)_{n+1}} - \frac{(q^2;q^2)_\infty}{(q;q^2)_\infty}\) 
 - \frac{(q^2;q^2)_\infty}{(q;q^2)_\infty}\( \frac{1}{4} +  \sum_{j=1}^\infty \frac{(-1)^j q^j}{1-q^j}\)$$
\end{enumerate}

\item[(C)] (Roots of unity)
 $M_1(e^{2\pi i x} )$ is finite exactly at  those $x$ in 
 $ 
 \{ \frac{a}{b} : (a, b) = 1, b \text{ is odd} \} = S_0$.   

$M_2(e^{2\pi i x} )$ is finite exactly at those $x$ in 
$\{ \frac{a}{b} : (a,b) = 1, b \text{ is even} \} = S_\infty$.

\end{theorem}

\begin{remark}
In each case, we see that the series $\cS[M_j]$ is the 
the ``half'' derivative of the Zagier shadow.  A quantum modular form results in an exactly analogous fashion 
to the ``half'' derivative of the Dedekind $\eta$-function, as discussed in Section \ref{subsec:quantum}. 
For more on ``half'' derivatives see \cite{zagierStrange}.  
\end{remark}

\begin{remark}
Recent work of Rolen and Schneider \cite{rs} demonstrates how to build a ``vector-valued'' quantum 
modular form from this pair of examples. 
\end{remark}

Notice that, modular forms are mock theta functions whose Zagier shadow is 0.  So 
we can extend this to a number of other examples to produce ``shadows'' and ``ghosts'' for 
series equal to modular forms

\subsection{Modular forms}\label{subsec:mod_forms}
The results of this subsection demonstrate renormalization on
$q$-series which represent modular forms.  We begin with two
examples. 

First, since both functions are the generating function for the
number of integer partitions:
$$P_1(q):= \sum_{n=0}^\infty \frac{q^n}{(q;q)_n}$$ equals $\prod_{m=1}^{\infty} (1-q^m)^{-1}$
when $\abs{q}<1$.  On the other hand, $P_1(q)$ converges for $\abs{q}>1$.  In particular, 
$$P_1(q^{-1}) = \sum_{n=0}^\infty \frac{(-1)^n q^{\frac{1}{2}n(n-1)}}{(q;q)_n}.$$
Consequentially, $P_1(q^{-1}) = 0$ (see, for instance, \cite{BFR}).  So, $\cS[P_1](q) = 0$ and 
$\cG[P_1](q) = 0$. 
 
There is a second well known $q$-hypergeometric series which equals
the partition generating function.  It is $$P_2(q):= \sum_{n=0}^\infty
\frac{q^{n^2}}{(q;q)_n^2}.$$ As before, we see that $P_2(q)$ converges
for $\abs{q}>1$ and
$$P_2(q^{-1}) = \sum_{n=0}^\infty \frac{q^n}{(q;q)_n^2} =
\frac{1}{(q;q)_\infty^2} \sum_{n=0}^\infty (-1)^n
q^{\frac{1}{2}n(n+1)},$$ 
a series which is not zero!  So $\cS[P_2](q) = 0$, but 
$\cG[P_2](q) =\frac{1}{(q;q)_\infty^2} \sum_{n=0}^\infty (-1)^n q^{\frac{1}{2}n(n+1)}.$
 Identities of
this form are the focus previous work of the third author with
Bringmann and Folsom \cite{BFR}.

Let 
\begin{align}
F_1(q):=& \sum_{n=0}^\infty F_{1, n}(q) = \sum_{n=0}^\infty \frac{(-1)^n q^{\frac{1}{2}n(n+1)}}{(q;q)_n}  \\
&= 1 - q - q^2 + q^5 + q^7 - q^{12} - q^{15} + q^{22} + q^{26} - q^{35} - q^{40} + q^{51} +  \cdots \nonumber 
\\ 
F_2(q):= & \sum_{n=0}^\infty F_{2, n}(q) = \sum_{n=0}^\infty \frac{(-1)^n (-q;q)_n}{(q;q)_{n}}  
= \frac{1}{2} (1-q) \sum_{n=0}^\infty \frac{(-q)^n (-q;q)_n}{(q;q)_n}
 \\
 &= \frac{1}{2} - q + q^4 - q^{9} + q^{16} - q^{25} + q^{36} - q^{49} + q^{64} - \cdots \nonumber \\
F_3(q):= & \sum_{n=0}^\infty F_{3, n}(q) = \sum_{n=0}^\infty \frac{q^{n} (q;q^2)_n}{(q^2;q^2)_n} \\
 & =  1 + q + q^3 + q^{6} + q^{10} + q^{15} + q^{21} + q^{28} + q^{36} + q^{45} +  q^{55} + \cdots   \nonumber \\
F_4(q):= & \sum_{n=0}^\infty F_{4,n}(q) = \sum_{n=0}^\infty \frac{(-1)^{n}q^{2n+1} (-q)_n}{(q)_n(1-q^{2n+1})} \\
&=  q + q^2 - q^4 + q^8 + q^9 - q^{16} + q^{18} + q^{25} + q^{32} - q^{36} + q^{49} + q^{50}  -q^{64} +  \cdots  \nonumber 
\end{align}
\begin{remark}
The series defining $F_2(q)$ does not technically converge for $\abs{q}<1$.  However, 
standard recurrence relations for the series extend converge into this domain. 
See \eqref{eqn:fine2.4}
and the discussion thereafter. 
\end{remark}

\begin{theorem}\label{thm:modular}
\begin{enumerate}
\item[(A)] For $\abs{q}< 1$ 
\begin{align*}
F_1(q) = &\sum_{n\in \Z} (-1)^n q^{\frac{1}{2}n(3n+1)} \\
F_2(q) =& \frac{1}{2} \sum_{n\in \Z} (-1)^n q^{n^2} \\
F_3(q) =& \frac{1}{2} \sum_{n\in \Z} q^{\frac{1}{2}n(n+1)} \\
F_4(q) =& \frac{1}{2} \sum_{n\in \Z}(-1)^{n+1} q^{n^2} + \frac{1}{2} \sum_{n\in \Z} q^{2n^2}
\end{align*}
In each case, $F_j(q)$ is, up to a power of $q$, a holomorphic modular form of weight $1/2$. 

\item[(B)] 
Take  $\cS[F_j](q) = 0$ for $j=1, 2, 3$ and 
\begin{align*}
\cG[F_1](q) =& 
\sum_{n=0}^\infty \(F_{1,n}(q^{-1}) - F_{1, \infty}(q^{-1})  \) \\
   = &\sum_{n=0}^\infty \(\frac{1}{(q;q)_n} - \frac{1}{(q;q)_\infty} \) =  -\frac{1}{(q;q)_\infty} \sum_{n=1}^\infty \frac{q^n}{1-q^n} \\
\cG[F_2](q) =& \sum_{n=0}^\infty \(F_{2,n}(q^{-1}) - F_{2, \infty}(q^{-1})  \) \\ = &
\sum_{n=0}^\infty \(\frac{(-q;q)_n}{(q;q)_n} - \frac{(-q;q)_\infty}{(q;q)_\infty} \) =
- 2\frac{(-q;q)_\infty}{(q;q)_\infty} \sum_{n=1}^\infty \frac{q^n}{1-q^{2n}} \\
\cG[F_3](q) = &
\sum_{n=0}^\infty \(F_{3,n}(q^{-1}) - F_{3, \infty}(q^{-1})  \) \\ = &
\sum_{n=0}^\infty \(\frac{(q;q^2)_n}{(q^2;q^2)_n} - \frac{(q;q^2)_\infty}{(q^2;q^2)_\infty} \) =
\frac{(q;q^2)_\infty}{(q^2;q^2)_\infty} \sum_{n=1}^\infty \frac{(-1)^n q^{n}}{1-q^n}.
\end{align*}
Additionally, $\cS[F_4] = 0$ and  
\begin{align*}
\cG[F_4](q) = \sum_{n=0}^\infty \(F_{4,n}(q^{-1}) - F_{4, \infty}(q^{-1})  \) = &
-\sum_{n=0}^\infty \(\frac{(-q;q)_n}{(q;q)_n(1-q^{2n+1})} - \frac{(-q;q)_\infty}{(q;q)_\infty} \) \\
 =& \frac{(-q;q)_\infty}{(q;q)_\infty}\(  2\sum_{n=1}^\infty \frac{q^{n}}{1-q^{2n}}  -h(q)\)
\end{align*}
where $$h(q):= \sum_{n=0}^\infty q^{n+1} \frac{(q)_{2n+1}}{(-q)_{2n+2}} \ \ \ \text{ and } \ \ \ 
h(q) = - \sum_{\fa \subset \Z[\sqrt{2}]} (-1)^{N(\fa)} q^{\abs{N(\fa)}}.$$

\end{enumerate}
\end{theorem}

\begin{remark}
The series $F_4(q)$ is different from the other cases considered in this work.  In each other 
case the ``ghost'' (i.e. the series $\cG[F_j](q)$ from \eqref{eqn:sumsoftails}) that arises from 
renormalizing is equal to 
$H_\infty(q^{-1})$ times a divisor-type function.
\end{remark}

\begin{remark}
The first identity of (B) is well known. See for instance Fine \cite{fine}.
\end{remark}




\subsection{Outline}
In Section \ref{sec:background} we collect some results on $q$-hypergeometric 
series.
Sections \ref{sec:mock}, and \ref{sec:modular} contain the proofs of 
Theorems \ref{thm:mock} and \ref{thm:modular}, respectively. 
Section \ref{sec:Recurrence} contains a survey of techniques for carrying out renormalization techniques. 
This section contains a number of additional results which may be of independent interest in the combinatorics of partitions and $q$-series. 
In Section \ref{sec:Other} we discuss ``renormalization''
 applied to a number of other examples which exhibit interesting 
structure, but do not seem to fit into any known theory of non-holomorphic modular forms. It would be interesting
if these examples were explored further by others.

\section{Preliminaries}\label{sec:background}

This section collects some results dealing with $q$-hypergeometric series. 
Throughout we adopt Fine's \cite{fine} notation for the basic hypergeometric series 
$$F(a, b;t) = F(a, b; t:q):= \sum_{n=0}^\infty \frac{(aq;q)_n}{(bq;q)_n} t^n.$$
The following identity is (6.3) of \cite{fine}
\begin{equation}\label{eqn:fine6.3}
F(a, b;t) = \frac{1-b}{1-t} F\( \frac{at}{b}, t;b\).
\end{equation}
Moreover,  when $\abs{t} = 1$, but $t\ne 1$ and $\abs{q}<1$ 
the series converges  via the recurrence 
\begin{equation}\label{eqn:fine2.4}
F(a, b;t) = \frac{1-b}{1-t} + \frac{b-atq}{1-t} F(a, b; tq);
\end{equation}
this identity appears as (2.4) in \cite{fine}. 

The following two equations are (12.2) and (14.31) of Fine's book \cite{fine}:
\begin{equation}\label{eqn:fine12.2}
(1-t) F(a, b;t) = \sum_{n=0}^\infty \frac{(b/a)_n}{(bq)_n(tq)_n} (-at)^n q^{\frac{1}{2}n(n+1)},
\end{equation}
and 
\begin{equation}\label{eqn:fine14.31}
(1-a)F(a,-a;a) = 1+ 2 \sum_{n=1}^\infty (-1)^n a^{2n} q^{n^2}.
\end{equation}
Additionally, equation (1) of \cite{rogers} is
\begin{equation}\label{eqn:rogers}
F(a/q, b/q;t) = \sum_{n=0}^\infty \frac{(a;q)_n (atq/b;q)_n b^n t^{n} q^{n^2 - n} (1-atq^{2n})}{(b;q)_n 
(t;q)_{n+1}}.
\end{equation}
The Baily-Daum summation formula (see (1.8.1) of \cite{GR}) is 
\begin{equation}\label{eqn:bailyDaum}
\sum_{n=0}^\infty \frac{(a)_n(b)_n}{(q)_n(aq/b)_n} \(-\frac{q}{b}\)^n = 
\frac{(-q;q)_\infty (aq;q^2)_\infty (aq^2/b^2;q^2)_\infty}{(aq/b)_\infty (-q/b)_\infty}
\end{equation}
One has the $q$-binomial theorem 
\begin{equation}\label{eqn:qbinomial}
\sum_{n=0}^\infty \frac{(a/b)_n}{(q)_n} b^n = \frac{(a)_\infty}{(b)_\infty}.
\end{equation}
Finally, the following identities are due to  Ramanujan:
\begin{equation}\label{eqn:ramanujanFalse}
\sum_{n=0}^\infty \frac{(-1)^n a^{2n} q^{\frac{1}{2}n(n+1)}}{(-aq;q)_n} = 
\sum_{n=0}^\infty a^{3n} q^{\frac{1}{2}n(3n+1)} (1-a^2 q^{2n+1}). 
\end{equation}
and
\begin{equation}\label{eqn:entry1.7.2}
\sum_{n=0}^\infty  \frac{(-aq/b)_n b^n (-1)^n q^{\frac{1}{2}n(n+1)}}{(-b;q)_{n+1}} = 
\sum_{n=0}^\infty \frac{(-b)^n (-q;q)_n (-aq/b;q)_n}{(aq;q^2)_{n+1}}.
\end{equation}
The first is (3.1) of \cite{BY} and  the second 
appears as Entry 1.7.2 of Ramanujan's Lost notebook book  volume II \cite{ABII}.

The next two theorems will be used to carry out the renormalization in the next three sections of this paper. 
The theorems are due to Andrews, Jim\'enez-Urroz, and Ono \cite{AJO}.
\begin{theorem}[Theorem 1 of \cite{AJO}]\label{thm:AJO}
\begin{align*}
\sum_{n=0}^\infty & \( \frac{(t;q)_\infty}{(a;q)_\infty} - \frac{(t;q)_n}{(a;q)_n}\) 
= \sum_{n=1}^\infty \frac{(q/a;q)_n}{(q/t;q)_n} \(\frac{a}{t}\)^n \\
& \hspace{.7in} + \frac{(t;q)_\infty}{(a;q)_\infty} \( \sum_{n=1}^\infty \frac{q^n}{1-q^n} + \sum_{n=1}^\infty 
\frac{q^nt^{-1}}{1-q^nt^{-1}} - \sum_{n=0}^\infty \frac{tq^n}{1-tq^n} - \sum_{n=0}^\infty \frac{at^{-1}q^n}{1-at^{-1}q^n}\).
\end{align*}
\end{theorem}

\begin{theorem}[Theorem 2 of \cite{AJO}]\label{thm:AJO2}
\begin{align*}
\sum_{n=0}^\infty & \( \frac{(a;q)_\infty(b;q)_\infty}{(q;q)_n (c;q)_\infty} - \frac{(a;q)_n(b;q)_n}{(q;q)_n (c;q)_n}\)  \\
&=  \frac{(a;q)_\infty (b;q)_\infty}{(q;q)_\infty (c;q)_\infty}
 \( \sum_{n=1}^\infty \frac{q^n}{1-q^n}  - \sum_{n=0}^\infty \frac{aq^n}{1-aq^n} - 
 \sum_{n=1}^\infty \frac{(c/b;q)_n b^n}{(a;q)_n (1-q^n)}\).
\end{align*}
\end{theorem}

\section{Weight $3/2$ mock theta functions}\label{sec:mock}
In this section we prove Theorem \ref{thm:mock}. 
We will need the following series 
\begin{align}
\ol{f}(q) =& 1+ 2\sum_{n=1}^\infty \frac{q^{\frac{1}{2}n(n+1)}}{(-q;q)_n(1+q^n)}  \\
=& 1 + 2q - 4q^2 + 8q^3 - 10q^{4} + 8q^{6} + 16 q^7 - 20q^8 + 10q^{9} - 8q^{10} + 24q^{11} - \cdots \nonumber \\ 
N^o(0,-1,1;q) =& \sum_{n=0}^\infty \frac{(-1)^n q^{(n+1)^2}}{(q;q^2)_{n+1} (1-q^{2n+1})} \\
=& q + 2q^2 + 3q^3 + 3q^4 + 4q^5 + 5q^6 + 4q^7 + 5q^8 + 7q^9 + 5q^{10} + 6q^{11} + 8q^{12} + \cdots  \nonumber 
\end{align}

The following is contained in Theorem 1.1 of \cite{BL1}. 
\begin{theorem}[Theorem 1.1 of \cite{BL1}]\label{thm:BL1}
With $q = e^{2\pi i z}$,  Let $\frac{\eta(z)^2}{\eta(2z)} = \sum_{n=0}^\infty b_n q^n$.  Then 
$$\ol{f}(q) - \frac{2}{\sqrt{\pi}} \sum_{n} b_n n^{\frac{1}{2}} \Gamma\(-\frac{1}{2}, 4\pi ny\) q^{-n}$$
is a non-holomorphic modular form of weight $3/2$. 
\end{theorem}
The following theorem is contained in Theorem 4.5 (3) of \cite{ABL}. 
\begin{theorem}[Theorem 4.5 (3) of \cite{ABL}]\label{thm:ABL}
With $q =e^{2\pi i z}$, let $\frac{\eta^2(2z)}{\eta(z)} = q^{\frac{1}{8}} \sum_{n=0}^\infty b_n q^n$.
Then 
$$q^{-\frac{1}{8}} N^o(0, -1, 1;q)  - \frac{1}{\sqrt{2\pi}} \sum_{n} b_n n^{\frac{1}{2}} 
 \Gamma\(-\frac{1}{2}, 4\pi ny\) q^{-n}$$
is a non-holomorphic modular form of weight $3/2$. 
\end{theorem}

We begin with a preliminary proposition 
\begin{proposition}\label{prop:mock_series}
In the notation from Section \ref{sec:results}, for $\abs{q}<1$,  
we have 
$$M_1(q) = \frac{1}{2} + \sum_{n=1}^\infty \frac{q^{\frac{1}{2}n(n+1)} }{(-q;q)_n (1+q^n)} = \frac{1}{2} \ol{f}(q)$$
and 
$$M_2(q) = \sum_{n=0}^\infty \frac{(-1)^{n+1} q^{(n+1)^2}}{(q;q^2)_{n+1} (1-q^{2n+1})} 
= - N^o(0, -1, 1;q).$$
\end{proposition}
\begin{proof}
Both results follow from \eqref{eqn:fine12.2}.
In the first case we take $a = 1$ and $b=t=-1$. The result follows immediately. 
In the second case we first send $q\mapsto q^2$, then take $a=1$, $b=q$, and 
$t = q$.  We obtain 
$$(1-q) \sum_{n=0}^\infty \frac{(q^2;q^2)_n q^n}{(q^3;q^2)_n} 
= \sum_{n=0}^\infty \frac{(q; q^2)_n}{(q^3;q^2)_n(q^3;q^2)_n} (-1)^n q^{n^2+2n}.$$ 
Multiplying both sides of the expression by $\frac{q}{(1-q)^2}$ gives the result. 
\end{proof}

\begin{proof}[Proof of Theorem \ref{thm:mock}]
The claims that $M_1(q)$ and $M_2(q)$ are mock theta functions follow from 
Proposition \ref{prop:mock_series} and 
Theorems \ref{thm:BL1} and \ref{thm:ABL}. 

The results concerning renormalization appear in Theorem 3 of \cite{AJO}.  
\end{proof}

\begin{remark}
It may seem there are other natural  choices for $\cS[M_2]$ and $\cG[M_2]$. In particular, 
$$\frac{(q^2;q^2)_\infty}{(q;q^2)_\infty} =  \sum_{n=0}^\infty q^{\frac{1}{2}n(n+1)} = \frac{1}{2} \sum_{n\in \Z} q^{\frac{1}{2} n(n+1)}$$ is a modular form. Thus we could have taken $\cS[M_2](q) = 
\sum_{n=0}^\infty (n+\alpha) q^{\frac{1}{2}n(n+1)}$ for any constant $\alpha$. However, 
since the ``half-derivative'' of $\sum_{n\in \Z} q^{\frac{1}{2}(n+\frac{1}{2})^2} = 
2\sum_{n=0}^\infty q^{\frac{1}{2}(n+\frac{1}{2})^2}$ is 
$\sum_{n=0}^\infty (2n+1) q^{\frac{1}{2}\( n+\frac{1}{2}\)^2}$ it is clear we should take the shadow to 
be given as in the theorem. 

There are similar choices for $\cS[M_1]$ because 
$$\frac{(q;q)_\infty}{(-q;q)_\infty} = \sum_{n\in \Z} (-1)^n q^{n^2}.$$
\end{remark}

\section{Weight $1/2$  theta functions}\label{sec:modular}
In this section we prove Theorem \ref{thm:modular}. 
The following result of  Bringmann and Kane 
\cite{BK2} is used to establish Theorem \ref{thm:modular}. 
\begin{theorem}\label{thm:BK2}
In the notation of Theorem \ref{thm:modular}
$$-\sum_{n=0}^\infty \( \frac{(-q)_n}{(q)_n (1-q^{2n+1})} - \frac{(-q)_\infty}{(q)_\infty} \) =  \frac{(-q)_\infty}{(q)_\infty} \( 2D_2(q) - h(q)\)$$
where $D_2(q) = \sum_{n=1}^\infty d_o(n) q^n$ and $d_o(n)$ is the number of odd divisors of $n$. 
\end{theorem}

\begin{proof}[Proof of Theorem \ref{thm:modular}]
We begin by proving the claims of (A).  The first equality follows from 
\eqref{eqn:ramanujanFalse} with $a = -1$.
From \eqref{eqn:fine14.31} we have 
$$ F_2(q) = F(-1, 1,-1) = \frac{1}{2} \( 1+ 2 \sum_{n=1}^\infty (-1)^n q^{n^2}\),$$
which establishes the second claim. 
By \eqref{eqn:rogers} we have 
\begin{align*}
F_3(q) = F(q\cdot q^{-2}, q^2\cdot q^{-2}; q: q^2) =&
 \sum_{n=0}^\infty \frac{(q;q^2)_n (q^2;q^2)_n q^{3n} q^{2n^2 - 2n} (1-q^{4n+2})}{(q^2;q^2)_n (q;q^2)_{n+1}}  \\
=&   \sum_{n=0}^\infty q^{2n^2 + n} (1+q^{2n+1}), 
\end{align*}
which establishes the third claim. 

The fourth claim is more intricate. We have 
\begin{align*}
\sum_{n=0}^\infty \frac{(-1)^n q^{2n+1} (-q)_n}{(q)_n (1-q^{2n+1})} = &
\sum_{m=0}^\infty q^{m+1} \sum_{n=0}^\infty \frac{(-q^{2(m+1)})^n (q)_n}{(-q)_n} \\
=& \sum_{m=0}^\infty q^{m+1} \frac{(q^{2m+3};q)_\infty}{(-q^{2m+2};q)_\infty} \\
=& \frac{(q)_\infty}{(-q)_\infty} \sum_{m=0}^\infty q^{m+1} \frac{(-q;q)_{2m+1}}{(q;q)_{2m+2}}
\end{align*}
where we have used \eqref{eqn:qbinomial} with $b = -q^{2m+2}$ and $a = q^{2m+3}$. 
Continuing in this fashion we see that 
\begin{align}
\sum_{n=0}^\infty \frac{(-1)^n q^{2n+1} (-q)_n}{(q)_n (1-q^{2n+1})} = &
\frac{(q)_\infty}{(-q)_\infty} \frac{1}{2} \sum_{m=0}^\infty q^{m+1} \frac{(-1;q^2)_{m+1} (-q;q^2)_{m+1}}{(q;q^2)_{m+1} (q^2;q^2)_{m+1} } \nonumber \\
=& -\frac{1}{2} \frac{(q)_\infty}{(-q)_\infty} + \frac{1}{2} \frac{(q)_\infty}{(-q)_\infty} 
\sum_{m=0}^\infty q^{m} \frac{(-1;q^2)_{m} (-q;q^2)_{m}}{(q;q^2)_{m} (q^2;q^2)_{m}}\nonumber \\
=& -\frac{1}{2}\frac{(q)_\infty}{(-q)_\infty} + \frac{1}{2} \frac{(q;q)_\infty (-q^2;q^2)_\infty (-q^2;q^4)_\infty^2}{(-q;q)_\infty (q;q^2)_\infty^2} \nonumber
\end{align}
where we have used \eqref{eqn:bailyDaum} with $q\mapsto q^2$, $a = -1$ and $b= -q$. 
Using 
$$(-q;q)_\infty = \frac{1}{(q;q^2)_\infty} = \frac{(q^2;q^2)_\infty}{(q;q)_\infty} \ \ \ 
\text{ and } \ \ \ 
(-q;q^2)_\infty = \frac{(q^2;q^2)_\infty^2}{(q;q)_\infty (q^4;q^4)_\infty}$$
we have 
$$\sum_{n=0}^\infty \frac{(-1)^n q^{2n+1} (-q)_n}{(q)_n (1-q^{2n+1})} = 
\frac{1}{2} \frac{(q)_\infty}{(-q)_\infty} + \frac{1}{2} \frac{(q^4;q^4)_\infty^5}{(q^2;q^2)_\infty^2
(q^8;q^8)_\infty^2}.$$
Now equation (2.2.12) of \cite{andrewsPartitions} gives 
$$\frac{(q)_\infty}{(-q)_\infty} = \sum_{n\in \Z} (-1)^n q^{n^2}.$$  Similarly, Jacobi's identity
gives
$$\frac{(q^2;q^2)^5}{(q;q)_\infty^2 (q^4;q^4)_\infty^2} = \sum_{n\in\Z} q^{n^2}.$$
Applying these gives the result. 

The first three 
claims of (B) follow from Theorem 3 of \cite{AJO}.  In particular, they are special cases of 
Theorem \ref{thm:AJO2}.   The claim for $F_4(q)$ follows from Theorem \ref{thm:BK2}. Moreover, 
the claims about the series $h(q)$ are proven in \cite{BK} (and recalled in \cite{BK2}). 
\end{proof}

\section{Techniques of  Renormalization}\label{sec:Recurrence}

Renormalization is an art.   We have primarily used the results of Andrews, Jimenez, and Ono \cite{AJO}
to establish the results of Section \ref{sec:Results}. 
This section contains some general discussion 
of renormalization techniques and gives a number of additional results concerning the series
studied in Section \ref{sec:Results}. 

\subsection{Some preliminaries about $(\sigma, \sigma^*)$}
For future reference we record some facts about the pair $(\sigma, \sigma^*)$. 
Andrews, Dyson, Hickerson, and Cohen established the following identities for $\sigma$ and $\sigma^*$ in  \cite{andrewsEuler, ADH, cohen}, which allows us to evaluate them for $q$ an arbitrary root of unity
\begin{align}
\sigma(q) = & 1+ \sum_{n=0}^\infty q^{n+1} (q-1) (q^2-1) \cdots (q^n-1) \label{eqn:sigma_rootsOfUnity}\\
\sigma^*(q) =& -2 \sum_{n=0}^\infty q^{n+1} (1-q^2) (1-q^4) \cdots (1-q^{2n}) \label{eqn:sigma*_rootsOfUnity}
\end{align}
Cohen observed that for every root of unity $q$, 
$$
\sigma(q^{-1}) = - \sigma^*(q).
$$
%

\subsection{Formal $q$-series manipulation}
The identities of Andrews, Jimmez, and Ono \cite{AJO} are proved using formal manipulations of 
identities arising from the theory of basic hypergeometric series. 

Andrews and Freitas established the following variants of Theorems \ref{thm:AJO} and \ref{thm:AJO2}. 
\begin{theorem}[Corollary 4.2 and Theorem 4.4 of \cite{AF}]\label{thm:AF}
\begin{equation}\label{eqn:AF_thm_1}
\sum_{n=0}^\infty \( \frac{(t)_n}{(a)_n} - \frac{(t)_\infty}{(a)_\infty}\) = 
\frac{(t)_\infty}{(a)_\infty} \sum_{n=1}^\infty \frac{(a/t)_n}{(q)_n} \frac{t^n}{1-q^n}
\end{equation}
\begin{equation}\label{eqn:AF_thm_2}
\sum_{n=0}^\infty \( \frac{(a)_n (b)_n}{(c)_n (q)_n} - \frac{(a)_\infty (b)_\infty}{(c)_\infty (q)_\infty}\)
= \frac{(a)_\infty (b)_\infty}{(c)_\infty (q)_\infty} \sum_{n=1}^\infty \( \frac{(q/b)_n b^n}{(q)_n} + \frac{(c/a)_n a^n}{(b)_n}\) \frac{1}{1-q^n}
\end{equation}
\end{theorem}
\begin{remark}
The second of these is equivalent to Theorem \ref{thm:AJO2} using 
$$\sum_{n=1}^\infty \frac{(q/b)_n b^n}{(q)_n}  = \sum_{n=1}^\infty \frac{q^n}{1-q^n} - \sum_{n=0}^\infty \frac{bq^n}{1-bq^n}.$$
\end{remark}

Direct applications of these series manipulations gives the following renormalizations for the series 
$F_1(q)$, $F_2(q)$ and  $F_3(q)$. 
{\small 
\begin{align*}
\sum_{n=0}^\infty \(F_{1,n}(q^{-1}) - F_{1, \infty}(q^{-1})  \) &= &\sum_{n=0}^\infty \(\frac{1}{(q;q)_n} - \frac{1}{(q;q)_\infty} \) =  \frac{1}{(q;q)_\infty} \sum_{n=1}^\infty \frac{(-1)^n q^{\frac{1}{2} n(n+1)}}{(q;q)_n(1-q^n)} \\
\sum_{n=0}^\infty \(F_{2,n}(q^{-1}) - F_{2, \infty}(q^{-1})  \) &=& 
   \sum_{n=0}^\infty \(\frac{(-q;q)_n}{(q;q)_n} - \frac{(-q;q)_\infty}{(q;q)_\infty} \) =
  \frac{(-q)_\infty}{(q)_\infty} \sum_{n=1}^\infty \frac{(-1)_n (-q)^n}{(q)_n (1-q^n)} \\
&& = \frac{(-q)_\infty}{(q)_\infty} \sum_{n=1}^\infty \frac{(-q)_n}{(q)_n} \frac{(-1)^n}{(1-q^n)} \\
&& =   \frac{(-q)_\infty}{(q)_\infty} \sum_{n=1}^\infty \( \frac{(-1)^n q^{\frac{1}{2}n(n+1)}}{(q;q)_n (1-q^n)}  + \frac{(-1)^n}{(1-q^n)} \)   \\
\sum_{n=0}^\infty \(F_{3,n}(q^{-1}) - F_{3, \infty}(q^{-1})  \) &=& 
\sum_{n=0}^\infty \(\frac{(q;q^2)_n}{(q^2;q^2)_n} - \frac{(q;q^2)_\infty}{(q^2;q^2)_\infty} \) =
\frac{(q;q^2)_\infty}{(q^2;q^2)_\infty} \sum_{n=1}^\infty \frac{(q;q^2)_n q^{n}}{(q^2;q^2)_n (1-q^{2n})}
\end{align*}
}
The first of these identities is obtained directly from either part of Theorem \ref{thm:AF}.  
Together with Theorem \ref{thm:mock} (B) the first identity  gives 
$$\sum_{n=1}^\infty \frac{(-1)^{n+1} q^{\frac{1}{2}n(n+1)}}{(q;q)_n (1-q^n)} = \sum_{n=1}^\infty \frac{q^n}{1-q^n}.$$
The next three lines are obtained from Theorem \ref{thm:AF}.  The first is the obtained from \eqref{eqn:AF_thm_1}
with $t = -q$ and $a = q$, the second is obtained from \eqref{eqn:AF_thm_2} with $a = c = 0$ and 
$b= -1$.  
The final of the three identities for $F_2$ is obtained from \eqref{eqn:AF_thm_2} with 
$a = -1$ and $b = c= 0$. 
The last equality is obtained from Theorem \ref{thm:AF} with $q\mapsto q^{2}$, $t = q$ and $a = q^2$. 
Together with Theorem \ref{thm:modular} (B), the final equality yields 
$$\sum_{n=1}^\infty \frac{(q;q^2)_n q^n}{(q^2;q^2)_n (1-q^{2n})} = \sum_{n=1}^\infty \frac{(-1)^{n-1} q^n}{1-q^n}.$$


\vspace{.1in}
There is  one further aspect of renormalization.  
Consider the example $\sum_{n=0}^\infty \frac{(-q;q)_n}{(q;q)_n}$ associated with $F_1(q^{-1})$.
In the above discussion to get convergence for $\abs{q}<1$  
the series $\sum_{n=0}^\infty \( \frac{(-q)_n}{(q)_n} - \frac{(-q)_\infty}{(q)_\infty}\)$
is introduced.  
However, there are many choices for how to carry out renormalization.  A different, natural choice is 
$$\sum_{n=0}^\infty \frac{1}{(q)_n}\( (-q)_n - (-q)_\infty\).$$ 
To calculate this renormalization the following theorem of 
Andrews and Freistas is useful. 
\begin{theorem}[Theorem 4.1 of \cite{AF}]\label{thm:AF2}
Let $g$ be a functions defined by the series $g(x) = \sum_{n=0}^\infty g_n x^n$. Then 
$$\sum_{n=0}^\infty g_n \( \frac{(t)_n}{(a)_n} - \frac{(t)_\infty}{(a)_\infty}\) = \frac{(t)_\infty}{(a)_\infty}
\sum_{n=1}^\infty \frac{(a/t)_n}{(q)_n} g(q^n) t^n.$$
\end{theorem} 
Thus using $\frac{1}{(x)_n} = \sum_{n=0}^\infty \frac{x^n}{(q)_n}$ we obtain 
$$\sum_{n=0}^\infty \frac{1}{(q)_n}\( (-q)_n - (-q)_\infty\) = \frac{(-q)_\infty}{(q)_\infty} \sum_{n=1}^\infty 
\frac{(-1)^n q^n}{1-q^n}.$$
This should be compared to the identity 
$$\sum_{n=0}^\infty \(\frac{(-q)_n}{(q)_n} - \frac{(-q)_\infty}{(q)_\infty} \) =
- 2\frac{(-q)_\infty}{(q)_\infty} \sum_{n=1}^\infty \frac{q^n}{1-q^{2n}} $$
of Theorem \ref{thm:modular}.


\subsection{Chapman's Combinatorial Lemma}
Another construction is the following Lemma of Chapman \cite{chapman2}.
\begin{lemma}\label{lem:chapman}
Let $a_n(q)$ be power series in the indeterminate $q$ with $a_n \to 0$ in the $q$-adic topology, 
then 
$$\sum_{N=0}^\infty \( \prod_{j=1}^\infty (1+a_j(q)) - \prod_{j=1}^N(1+a_j(q))\) = 
\sum_{n=1}^\infty n a_n(q) \prod_{j=1}^{n-1} (1+a_j(q)).$$
\end{lemma}
Chapman used this lemma to give  combinatorial proofs of some of `renormalization' identities. 

Recall, $\sigma(q) = \sum_{n=0}^\infty \sigma_n(q)$ with $\sigma_n(q) = \frac{q^{\frac{1}{2}n(n+1)}}{(-q;q)_n}$ so
$$\sum_{n=0}^\infty \(  \sigma_n(q^{-1})-  \frac{1}{(-q)_\infty} \) = \sum_{n=1}^\infty \frac{nq^n}{(-q)_n}.$$ 
\begin{remark}
Together with 
\begin{align}
\label{eqn:rama1}
\sum_{n=0}^\infty \( \sigma_n(q^{-1}) - \sigma_\infty(q^{-1})\) &= \sum_{n=0}^\infty \(\frac{1}{(-q;q)_n} - \frac{1}{(-q;q)_\infty}\) \\
 &= 
2\sum_{n=1}^\infty \frac{(-1)^{n-1} q^{n^2}}{(q;q^2)_n}
-\frac{1}{(-q;q)_\infty} \(\sum_{n=1}^\infty \frac{q^n}{1-q^n}\) \nonumber
\end{align}
(for the second equality, see for instance (3.27) of \cite{AGL})
 we obtain 
$$\sum_{n=1}^\infty n q^{n}(-q^{n+1};q)_\infty = (-q)_\infty \sigma^*(q) - \sum_{n\ge 1} \frac{q^n}{1-q^n}.$$
This has a nice combinatorial interpretation which will be discussed in a forthcoming work of the authors. 
\end{remark}

Recall, $\sigma^*(q) = \sum_{n=0}^\infty \sigma^*_n(q)$ with 
$\sigma_n^*(q) = \frac{(-1)^{n+1} q^{(n+1)^2}}{(q;q^2)_{n+1}}$ so 
$$\sum_{n=0}^\infty \( \sigma_n^*(q^{-1}) - \frac{1}{(q;q^2)_\infty}\) = \sum_{n=1}^\infty \frac{n q^{2n-1}}{(q;q^2)_n}.$$ 
This is also given in \cite{andrewsEuler}  (see (2.4)).

Recall, $W(q):= \sum_{n=0}^\infty W_n(q)$ with $W_n(q) = (-q)^n \frac{(q;q^2)_n}{(-q^2;q^2)_n}$
so that 
$$\sum_{n=0}^\infty \( W_n(q^{-1}) - \frac{(q;q^2)_\infty}{(-q^2;q^2)_\infty}\) 
= \sum_{n=1}^\infty \frac{nq^{2n-1} (q;q^2)_{n-1}}{(-q^2;q^2)_n} $$

Recall, $M_1(q):= \sum_{n=0}^\infty M_{1,n}(q)$ with $M_{1,n}(q) = \frac{(q)_n}{(-q)_n}$ so that 
$$\sum_{n=0}^\infty \( M_{1,n}(q^{-1}) - \frac{(q)_\infty}{(-q)_\infty}\) = 
-2\sum_{n=1}^\infty \frac{n q^n (q)_{n-1}}{(-q)_n}.$$ 

Recall, $F_1(q):= \sum_{n=0}^\infty F_{1,n}(q)$ with $F_{1, n}(q) = \frac{(-1)^n q^{\frac{1}{2}n(n+1)}}{(q;q)_n}$
so that 
$$\sum_{n=0}^\infty \( F_{1,n}(q^{-1}) - \frac{1}{(q)_\infty}\) = 
\sum_{n=1}^\infty \frac{n q^n}{(q)_n}.$$ 
\begin{remark} Consequentially, 
together with Theorem \ref{thm:modular} this yields 
$$\sum_{n=1}^\infty nq^n(q^{n+1};q)_\infty = \sum_{n=1}^\infty \frac{q^n}{1-q^n}.$$
\end{remark}

Recall, $F_2(q):= \sum_{n=0}^\infty F_{2,n}(q)$ with $F_{2,n}(q) = (-1)^n \frac{(-q)_n}{(q)_n}$
so that 
$$\sum_{n=0}^\infty \( F_{2,n}(q^{-1}) - \frac{(-q)_\infty}{(q)_\infty}\) 
= 2\sum_{n=1}^\infty \frac{nq^{n}(-q)_n}{(q)_n}.$$

Recall, $F_3(q):= \sum_{n=0}^\infty F_{3,n}(q)$ with $F_{3,n}(q) = q^n \frac{(q;q^2)_n}{(q^2;q^2)_n}$
so that 
$$\sum_{n=0}^\infty \( F_{3,n}(q^{-1}) - \frac{(q;q^2)_\infty}{(q^2;q^2)_\infty}\) 
= - (1-q)\sum_{n=1}^\infty \frac{nq^{2n-1}(q;q^2)_{n-1}}{(q^2;q^2)_n}.$$

\begin{remark}
The series $F_4(q)$ does not fit into the framework of Chapman's lemma. 
\end{remark}

\subsection{Recursion}
Zagier proved 
$$\sum_{n=0}^\infty \( (q)_n - (q)_\infty \) = \sum_{n=0}^\infty \( \frac{12}{n}\)  n q^{\frac{n^2-1}{24}}
  + (q)_\infty \sum_{n=1}^\infty \frac{q^n}{1-q^n}.$$
by introducing auxiliary variables and deducing the identity 
$$S(x) := \sum_{n=0}^\infty (x)_{n+1} x^n = \sum_{n=0}^\infty  \(\frac{12}{n}\) x^{\frac{1}{2}(n-1)} q^{\frac{1}{24} (n^2-1)}.$$
This auxiliary identity is proved by showing that both sides satisfy the recurrence 
$S(x) = 1-qx^2 - q^2 x^3 S(qx)$ (see exercise 10 in Chapter 2 of \cite{andrewsPartitions}). 

In this section recursion techniques are used to obtain identities for the renormalizations of 
$\sigma(q)$, $\sigma^*(q)$, and $F_1(q)$ from  Section \ref{sec:Results}.  The examples considered in this section are exactly those from Section \ref{sec:Results} with one $q$-Pochammer symbol. 
We have not considered those series with more than one $q$-Pochammer symbol, 
but it would be interesting to adapt these methods to that situation. 

\begin{theorem}\label{thm:recursion}
\begin{align}
\sum_{n=0}^\infty \(\frac{1}{(-q)_n} - \frac{1}{(-q)_\infty} \) 
=&2  \sum_{n=1}^\infty \frac{(-1)^{n-1} q^{n^2}}{(q^2;q^2)_n} \sum_{j=1}^n \frac{1}{1-q^{2j}}
 - \frac{1}{(-q)_\infty } \sum_{n=1}^\infty \frac{q^n}{1+q^n} \label{eqn:sigma_rec}\\
 \sum_{n=0}^\infty \(\frac{1}{(q;q^2)_n} - \frac{1}{(q;q^2)_\infty} \)
 =& - \sum_{n=1}^\infty \frac{q^{\frac{1}{2} n(n+1)}}{(q)_{2n}} \sum_{j=1}^{2n} \frac{1}{1-q^j}
  + \frac{1}{(q;q^2)_\infty} \sum_{n=1}^\infty \frac{q^{2n-1}}{1-q^{2n-1}}. \label{eqn:sigma^*_rec}\\
  \sum_{n=0}^\infty \(\frac{1}{(q)_n} - \frac{1}{(q)_\infty} \) 
=&-2  \sum_{n=1}^\infty \frac{q^{n^2}}{(q;q)_n^2} \sum_{j=1}^n \frac{1}{1-q^{j}}
 + \frac{1}{(q)_\infty } \sum_{n=1}^\infty \frac{q^n}{1-q^n} \label{eqn:F1_rec}
\end{align}
\end{theorem}
\begin{proof}
To prove \eqref{eqn:sigma_rec} introduce the two three variable series 
$$S_\sigma(a, x):= \sum_{n=0}^\infty \frac{x^n}{(-a)_n} \ \ \text{ and } \ \ 
S_\sigma(x) := \frac{1}{1+x} S_\sigma(xq,x)  = \sum_{n=0}^\infty \frac{x^n}{(-x)_{n+1}}.$$
The following are easily verified
\begin{align}
&(1-x) S_\sigma(a,x) = 1-\frac{a}{q} \( S_\sigma(a, xq) - 1\) \label{eqn:sigma1} \\
&S_\sigma(a,x) = 1 + \frac{x}{1+a} S(aq, x) \label{eqn:sigma2} \\
& S_\sigma(x) = \frac{1}{1-x^2} \frac{1}{(-xq)_\infty} + \frac{1}{1+x} \sum_{n=0}^\infty \( \frac{1}{(-xq)_n} - 
\frac{1}{(-xq)_\infty}\) x^n \label{eqn:sigma3}
\end{align}
Using \eqref{eqn:sigma1} and \eqref{eqn:sigma2} one obtains 
$$S_\sigma(x) = \frac{1}{1-x^2} \( 1 - qx^2 S_\sigma(xq) \).$$
Combining this with \eqref{eqn:sigma3} yields
$$\sum_{n=0}^\infty \( \frac{1}{(-xq)_n} - \frac{1}{(-xq)_\infty}\) x^n = 
\frac{1}{1-x} \( \sum_{n=0}^\infty \frac{(-1)^n q^{n^2} x^n }{(x^2q^2;q^2)_n} 
- \frac{1}{(-xq)_\infty}\).$$ 
Letting $x = 1-\epsilon$ and picking off the $\epsilon^0$ term of both sides gives the result. 
It is useful to note that 
$$\frac{(-1)^n x^n q^{n^2}}{(x^2q^2;q^2)_n} = \frac{(-1)^n q^{n^2} (1-2n\epsilon)}{(q^2;q^2)_n} 
\prod_{j=1}^n \( 1- \epsilon \frac{q^{2j}}{1-q^{2j}}\)$$ 
and 
$$\frac{1}{(-xq;q)_\infty} = \frac{1}{(-q)_\infty }\( 1+ \epsilon \sum_{j=1}^\infty \frac{q^j}{1+q^j} + O(\epsilon^2)\).$$
\begin{remark}
This yields 
$$\sum_{n=0}^\infty \frac{(-1)^n q^{n^2}}{(q^2;q^2)_n} = \frac{1}{(-q)_\infty}.$$
\end{remark}

\vspace{.1in}
To prove \eqref{eqn:sigma^*_rec} introduce the three variable series 
$$S_{\sigma^*}(a, x) := \sum_{n=0}^\infty \frac{x^n}{(aq;q^2)_n} \ \ \text{ and } \ \ 
 S_{\sigma^*}(x) := S_{\sigma^*}(x,x).$$
 The following are easily verified 
\begin{align}
&(1-x) S_{\sigma^*}(a,x) = 1+ \frac{a}{q} \( S_{\sigma^*}(a, xq^2) - 1\) \\
& S_{\sigma^*}(a,x) = 1 + \frac{x}{1-aq} S_{\sigma^*}(aq^2, x) \\
& S_{\sigma^*}(x) = \sum_{n=0}^\infty \( \frac{1}{(xq;q^2)_n} - \frac{1}{(xq;q^2)_\infty} \) x^n 
+ \frac{1}{1-x} \cdot \frac{1}{(xq;q^2)_\infty} \label{eqn:sigma^*3}
\end{align}
Combining these with $a =x$ yields
$$S_{\sigma^*}(x) = \frac{1}{1-x} \( 1 + \frac{x^2 q}{1-xq} S_{\sigma^*}(xq^2)\).$$
This yields
$$S_{\sigma^*}(x) = \sum_{n=0}^\infty \frac{x^{2n} q^{\frac{1}{2} n(n+1)}}{(x)_{2n+1}}.$$
Together with \eqref{eqn:sigma^*3} yields 
$$ \sum_{n=0}^\infty \( \frac{1}{(xq;q^2)_n} - \frac{1}{(xq;q^2)_\infty} \) x^n 
= \sum_{n=0}^\infty \frac{x^{2n} q^{\frac{1}{2} n(n+1)}}{(x)_{2n+1}}
-  \frac{1}{1-x} \cdot \frac{1}{(xq;q^2)_\infty}.$$
As before, set $x = 1-\epsilon$ and  compare the $\epsilon^0$ term of each side to obtain the desired 
identity. 
\begin{remark}
Notice that 
$$\sum_{n=0}^\infty \frac{q^{\frac{1}{2}n(n+1)}}{(q)_{2n}} = \frac{1}{(q;q^2)_\infty}.$$
\end{remark}

\vspace{.1in} 
To prove \eqref{eqn:F1_rec} introduce the three variable series 
$$S_{F_1}(a, x):= 
 \sum_{n=0}^\infty \frac{x^n}{(a)_n} \ \ \text{ and } \ \ 
S_{F_1}(x) := \frac{1}{1-x} S_{F_1}(xq,x)  = \sum_{n=0}^\infty \frac{x^n}{(x)_{n+1}}.
$$
The following are easily verified
\begin{align}
&(1-x) S_{F_1}(a,x) = 1+\frac{a}{q} \( S_{F_1}(a, xq) - 1\) \label{eqn:F11} \\
&S_{F_1}(a,x) = 1 + \frac{x}{1-a} S_{F_1}(aq, x) \label{eqn:F12} \\
& S_{F_1}(x) = \frac{1}{(1-x)} \frac{1}{(x)_\infty} + \frac{1}{1-x} \sum_{n=0}^\infty \( \frac{1}{(xq)_n} - 
\frac{1}{(xq)_\infty}\) x^n \label{eqn:F13}
\end{align}
Combining the first two we obtain 
$$S_{F_1}(x) = \sum_{n=0}^\infty \frac{x^{2n} q^{n^2}}{(x)_{n+1}^2}.$$
Equating this with \eqref{eqn:F13} yields
\begin{equation}\label{eqn:SF1}
\sum_{n=0}^\infty \( \frac{1}{(xq)_n} - \frac{1}{(xq)_\infty}\) x^n = \frac{1}{1-x}\( \sum_{n=0}^\infty 
\frac{x^{2n} q^{n^2}}{(xq)_n^2} - \frac{1}{(xq)_\infty}\).
\end{equation}
\end{proof}

The series in \eqref{eqn:SF1} has interesting connections with the literature.  The next two remarks 
are based on this observation. 
\begin{remark}
Substituting $x =-1$ yields the following curious identity for Ramanujan's third order mock theta function 
$$f(q) := \sum_{n=0}^\infty\frac{q^{n^2}}{(-q)_n^2} = \frac{1}{(-q)_\infty} + 2 
\sum_{n=0}^\infty \( \frac{1}{(-q)_n} - \frac{1}{(-q)_\infty}\) (-1)^n.$$
This is particularly curious in light of the fact that 
for $\abs{q}<1$, (see \cite{fine}) 
$$f(q) = 2 - \sum_{n=0}^\infty \frac{(-q)^n}{(-q)_n} = 2 - F(0, -1; -q) 
= F(0, -1, -1) = 2 \sum_{n=0}^\infty \frac{(-1)^n}{(-q)_n}$$
where the last equality follows by \eqref{eqn:fine2.4}.  The final 
series appears in the equation above if all occurrences of $1/(-q)_\infty$ are deleted
and is consistent with $\sum_{n=0}^\infty (-1)^n = \frac{1}{2}$. %
\end{remark}
\begin{remark}
The series 
\begin{align*}
\cF(x;q):=&  \frac{1}{(1-x)} \sum_{n=0}^\infty \frac{x^{2n} q^{n^2}}{(xq)_n^2} = (1-x) S_{F_1}(x) \\
 =& 1 + x + (1+q) x^2 + (1+q+2q^2) x^3 + (1+q+2q^2 + 3q^3) x^4 + O(x^5) \\
 =: & 
 \sum_{N=0}^\infty P_N(q) x^N 
 \end{align*}
arises in Chapter 9 of Andrews \cite{andrewsCBMS} (see also Sills \cite{sills}). 
In the context of those works there are three properties that are important.
\begin{enumerate}
\item $\cF(x;q)$ satisfies a first order nonhomogenous  $q$-difference equation. 
\item The polynomials $P_N(q)$ converge to $P(q) := \sum_{n=0}^\infty \frac{q^{n^2}}{(q)_n^2}$, the series
which $\cF(x;q)$ is a two-variable version of. 
\item $\lim_{x\to 1} (1-x) \cF(x;q) = P(q)$. 
\end{enumerate}
The identity \eqref{eqn:SF1} yields the second of these properties easily. 
Define the series $T_N(q)$ by  
$$\sum_{n=0}^\infty \( \frac{1}{(xq)_n} - \frac{1}{(xq)_\infty}\) x^n = \sum_{n=0}^\infty T_N(q) x^n.$$
It is clear that $T_N(q) = O(q^{N})$.  Finally, 
$$\frac{1}{(x)_\infty} = \sum_{N=0}^\infty \frac{1}{(q)_n} x^N.$$
From which it is evident that 
$$\lim_{N\to \infty} P_N(q) = \frac{1}{(q)_\infty} = \sum_{n=0}^\infty \frac{q^{n^2}}{(q)_n^2}.$$

Andrews shows that $P_N(q) = \sum_{j=0}^{N-1} q^j \binom{n-1}{j}_q$ where $\binom{A}{B}_q = \frac{(q)_A}{(q)_{B} (q)_{A-B}}$ is the $q$-binomial coefficient. 
\end{remark}

\vspace{.1in}
The identities  in Theorem \ref{thm:recursion} are similar to those obtained via bilateral summation, 
a formal $q$-series construction. 
Because of the similarities we record the results here. 
\begin{theorem}\label{thm:bilateral}
Recall from \eqref{eqn:sigma*_rootsOfUnity}
$$-\frac{1}{2} \sigma^*(q) = \sum_{n=0}^\infty q^{n+1} (q^2;q^2)_n.$$ For $\abs{q}<1$
\begin{align*}
-\frac{1}{2} \sigma^*(q) =& -\sum_{n=1}^\infty \frac{(-1)^n q^{n^2}}{(q^2;q^2)_n} \sum_{j=1}^n \frac{1}{1-q^{2j}}  + (q;q^2)_\infty  \sum_{n=1}^\infty \frac{q^{2n}}{1-q^{2n}}\\
=& \frac{1}{2} \sum_{n=1}^\infty \frac{(-1)^n q^{n^2}}{(q^2;q^2)_n} \sum_{j=1}^n \frac{q^j}{1-q^j} +  
\frac{(-q)_\infty}{(q)_\infty^2} \( \sum_{n=1}^\infty \frac{q^n}{1-q^n} + \sum_{n\in \Z} \frac{n (-1)^n q^{\frac{1}{2} n(n-1)}}{1+q^{n-1}}\).  
\end{align*}
\end{theorem}
\begin{remark}
The first of these results, together with \eqref{eqn:sigma_rec} of Theorem \ref{thm:recursion} establishes Ramanujan's 
identity \eqref{eqn:rama1}.
\end{remark}
\begin{proof}[Proof of Theorem \ref{thm:bilateral}]
By Theorem 1 of Choi \cite{choi} or Entry 3.3.1 of \cite{ABLost}, we have 
$$\sum_{n=1}^\infty q^{n} (w^{-1} q^2;q^2)_{n-1} = \frac{1}{1-w^{-1}} \( - \sum_{n=0}^\infty \frac{(-1)^n w^n q^{n^2}}{(wq^2;q^2)_n} + \frac{(wq;q^2)_\infty (w^{-1}q;q^2)_\infty (q^2;q^2)_\infty}{(wq^2;q^2)_\infty (q;q^2)_\infty} \) $$
(Note that 
$\sum_{n=0}^\infty \frac{(-1)^n q^{n^2}}{(q^2;q^2)_n} = (q;q^2)_\infty = \frac{1}{(-q)_\infty}$.)
Letting $w \to 1$ yields the first result.

Theorem 4 of \cite{choi} or page 67 of \cite{ABLost} 
yields
$$\sum_{n=0}^\infty q^{n+1} (-q)_n (t^{-1}q)_n - \frac{t}{2(1-t)} \sum_{n=0}^\infty 
\frac{(-t)^n q^{n^2}}{(-q)_n (tq)_n}  = \frac{(-q)_\infty}{(t)_\infty (q)_\infty} \sum_{n\in \Z} \frac{(-t)^n q^{\frac{1}{2} n(n-1)}}{1+q^{n-1}}.$$
As above, letting $t = 1-\epsilon$ and taking the $\epsilon^0$ term on each side yields the second claim. 
\end{proof}

\vspace{.1in}
There are two other series which were not considered previously in this paper, but 
we record now.

\begin{theorem} 
For $\abs{q}<1$ 
{\small 
\begin{align}
\sum_{n=0}^\infty \((-q)_n - (-q)_\infty \)  \nonumber 
=&  \sum_{n=0}^\infty \frac{(-1)_n q^{\frac{3n^2-n}{2}}}{(q)_n}  \\ 
  \nonumber & \ \ \  \times \( (1+q^{2n}) \sum_{j=1}^n \frac{2q^j}{q^{2j}-1} - 2q^{2n} 
-3n(1+q^{2n}) - \frac{ (1-q^n)(1+q^{2n})}{2(1+q^n)}\)  \\ & -2 (-q)_\infty \sum_{n=0}^\infty \frac{q^n}{1+q^n}
\\ 
\nonumber  \sum_{n=0}^\infty  \( (q;q^2)_n - (q;q^2)_\infty \)  
  =& \sum_{n=0}^\infty \frac{(-1)^n q^{3n^2} (q;q^2)_n}{(q^2;q^2)_n} \( (1-q^{4n+1}) \sum_{j=1}^{2n} \frac{(-1)^{n+1}}{1-q^{j}} +  (3n+2) q^{4n+1} - 3n\) \\
  & + (q;q^2)_\infty \sum_{n=0}^\infty \frac{q^{2n+1}}{1-q^{2n+1}}.
\end{align}
}
\end{theorem}
\begin{proof}[Sketch of proof]
These are proved analogously to the previous theorem using the series 
$$S_1(a,x): = \sum_{n=0}^\infty (-a)_n x^n \ \ \ \text{ and }  \ \ \ S_2(a,x) = \sum_{n=0}^\infty (aq;q^2) x^n.$$
\end{proof}

These should be compared with Ramanujan's identity (see \cite{andrewsEuler})
$$\sum_{n=0}^\infty ( (-q)_n - (-q)_\infty ) = 
-\frac{1}{2} \sum_{n=0}^\infty \frac{q^{\frac{1}{2}n(n+1)}}{(-q)_n} + (-q)_\infty \( \frac{1}{2} - \sum_{n=1}^\infty \frac{q^n}{1-q^n}\)$$
and 
$$\sum_{n=0}^\infty ( (-q)_n - (-q)_\infty ) = \sum_{n=1}^\infty n q^n (-q)_{n-1}$$
obtained from Chapman's Lemma \ref{lem:chapman} (see also (2.3) of \cite{andrewsEuler}). 

The second of these should be compared to 
$$\sum_{n=0}^\infty ((q;q^2)_n - (q;q^2)_\infty) = \sum_{n=1}^\infty \frac{(-1)^{n+1} q^{n^2}}{(q;q^2)_n}
 - (q;q^2)_\infty \sum_{n=1}^\infty \frac{q^n}{1-q^n} = \frac{1}{2} \sigma^*(q)  - (q;q^2)_\infty \sum_{n=1}^\infty \frac{q^n}{1-q^n}$$
 obtained from Theorem \ref{thm:AJO}. 
 Theorem \ref{thm:AJO2} gives 
\begin{align*}
 \sum_{n=0}^\infty ((q;q^2)_n - (q;q^2)_\infty) =& (q;q^2)_\infty \sum_{n=1}^\infty \frac{q^{n}}{(q^2;q^2)_n (1-q^{2n})} \\
  =& (q;q^2)_\infty \(  \sum_{n=1}^\infty \frac{(-1)^{n+1} q^{n}}{1-q^{n}} +  \sum_{n=1}^\infty \frac{q^{2n}}{(q;q^2)_n (1-q^{2n})}\).
\end{align*}

\section{Open Questions and Other Series}\label{sec:Other}
This section contains some additional examples which do not seem to fit as nicely into the theory of Maass waveforms
and mock theta functions. 

\subsection{Other series associated with modular forms and mock modular forms}

Resembling Zagier's strange function 
$F(q) = \sum_{n=0}^\infty (q)_n$ is the series $\sum_{n=0}^\infty (-1)^n (q)_n$. 
By Theorem \ref{thm:AF2} with $a = 0$ and $g_n = (-1)^n$ 
$$\sum_{n=0}^\infty (-1)^n ((q)_n - (q)_\infty) = (q)_\infty \sum_{n=1}^\infty \frac{q^n}{(q)_n (1+q^n)} 
 = q + q^3 - q^4 - q^8 + q^{10} - q^{13} + q^{14} + q^{15} +\cdots$$
 \begin{remark}
 This should be compared to Zagier's series for $\sum_{n=0}^\infty ((q)_n - (q)_\infty)$. 
 \end{remark}

The series on the right hand side is closely related to a famous ``partial theta function.''
Let 
$$G(a;q) := \frac{1}{2} + \sum_{n=1}^\infty \frac{q^n (-q)_{n-1}}{(a^{-1}q^2;q^2)_n}.$$
Note that  for $\abs{q}<1$
$$G(a;q^{-1}) = F(a; q):=  
\sum_{n=1}^\infty \frac{(-a)^n q^{\frac{1}{2} n(n+1)} (-q)_{n-1}}{(aq^2;q^2)_n} = \frac{1}{2} +  \sum_{n=1}^\infty (-a)^n q^{n^2}$$
For $a = 1$ or $a = -1$ the right hand side is a modular form.
Moreover, $$G(1;q) =  \frac{1}{(q)_\infty} \sum_{n=0}^\infty (-1)^n \( (q)_n - (q)_\infty\).$$
\begin{remark}
 Moreover, $G(1;q)$ is the generating function of the number of partitions of $n$ in which the least part is odd. 
 Equivalently, it is the generating function for the number of partitions of $n$ with the largest  part occurring 
 an odd number of times.  This will be discussed in a forthcoming work of the authors \cite{LNR}. 
\end{remark}
 
The series 
$$\sum_{n=0}^\infty (-1)^n (q;q)_n$$ does not need renormalization to continue.  
In fact, by \eqref{eqn:fine2.4} it is given by 
$$\frac{1}{2} \( 1+ \sum_{n=0}^\infty (-1)^n q^{n+1} (q)_n\)$$
(see also the proof of Corollary 4.3 in \cite{cohen}).
Therefore,  to infinite order near every root of unity 
$\sum_{n=0}^\infty (-1)^n ( (q)_n - (q)_\infty)$ agrees with this series, but by 
\eqref{eqn:sigma_rootsOfUnity}, this is $\frac{1}{2} \sigma(q)$. 

\begin{challenge*}
Provide a Hecke-type sum for $\sum_{n=0}^\infty (-1)^n ((q)_n - (q)_\infty)$. 
\end{challenge*}
\begin{challenge*}
Does the series $G(a;q)$ have any modular, mock modular, or quantum modular properties for other specializations of $a$?
\end{challenge*}

Andrews \cite{andrewsOrth} found a number of series related to Ramanujan's third order mock theta function 
$$\psi(q):= \sum_{n=0}^\infty \frac{q^{n^2}}{(q;q^2)_n} = \frac{1}{(q)_\infty} 
\sum_{n=0}^\infty  (-1)^n q^{2n^2+n} (1-q^{6n+6}) \sum_{j=0}^n q^{-\frac{1}{2}j(j+1)}.$$
They are the series 
\begin{align*}
a(q):=& \sum_{n=0}^\infty \frac{q^{n^2} (-q)_n^2}{(q)_{2n}} 
   = \frac{1}{(q)_\infty}  \sum_{n=0}^\infty q^{2n^2+n} (1-q^{6n+6}) \sum_{j=0}^n q^{-\frac{1}{2}j(j+1)}\\
b(q):=& \sum_{n=0}^\infty \frac{q^{n^2} (-q^2;q^2)_n}{(q)_{2n}}  = \frac{(-q;q^2)_\infty}{(q;q^2)_\infty} \psi(-q) \\
c(q):=& \sum_{n=0}^\infty \frac{q^{(n+1)^2 }(-q^2;q^2)_n}{(q)_{2n+1}} = \frac{(-q;q^2)_\infty}{(q;q^2)_\infty} (1-\psi(-q))
\end{align*}
where the second equality in each line is for $\abs{q}<1$. 

The series 
$$\psi(q^{-1}) = \sum_{n=0}^\infty \frac{(-1)^n}{(q;q^2)_n} = F(0, q^{-1}; -1:q^2) = \frac{1}{2} \( 1- \sum_{n=0}^\infty \frac{(-q)^n}{(q;q^2)_{n+1}}\)$$
by \eqref{eqn:fine2.4}.
Finally we have the partial theta identity 
$$\sum_{n=0}^\infty \frac{(-q)^n}{(q;q^2)_{n+1}} = \sum_{n=0}^\infty (-1)^n q^{6n^2+4n} (1+q^{4n+2}).$$
Such identities are typical.  See for example, the work of the Bringmann, Folsom, and the third author \cite{BFR}. 

It is curious to observe that carrying out renormalization we have 
$$\sum_{n=0}^\infty (-1)^n \( \frac{1}{(q;q^2)_n} - \frac{1}{(q;q^2)_\infty}\) = 
\frac{1}{(q;q^2)_\infty} \sum_{n=1}^\infty \frac{(-1)^n q^{n^2}}{(q^2;q^2)_n (1+q^{2n})} =  - q - q^2 - q^3 - q^4 + \cdots - 27q^{19} + \cdots $$
by Theorem \ref{thm:AF2}. 
\begin{challenge*}
Relate the series $$\sum_{n=1}^\infty \frac{(-1)^n q^{n^2}}{(q^2;q^2)_n (1+q^{2n})} = -q+q^4 - q^5 + q^6+q^8 + \cdots 
+ 89q^{54} - 93 q^{55} + 104 q^{56} + \cdots  $$
to the partial theta function appearing in $\psi(q^{-1})$. 
\end{challenge*}

Each of the series 
$a(q^{-1}) = \sum_{n=0}^\infty \frac{(-q)_n^2}{(q)_{2n}}$, $b(q^{-1}) = \sum_{n=0}^\infty \frac{(-q^2;q^2)_n}{((q)_{2n}}$, 
and $c(q^{-1}) = -\sum_{n=0}^\infty \frac{ (-q^2;q^2)_{n}}{(q)_{2n+1}}$ require renormalization. 
\begin{challenge*}
Give explicit formulas for $\sum_{n=0}^\infty \( \frac{(-q)_n^2}{(q)_{2n}}-  \frac{(-q)_\infty^2}{(q)_{\infty}}\)$, 
$\sum_{n=0}^\infty \( \frac{(-q^2;q^2)_n}{((q)_{2n}}  - \frac{(-q^2;q^2)_\infty}{(q)_\infty}\)$, 
and $\sum_{n=0}^\infty \( \frac{ (-q^2;q^2)_{n}}{(q)_{2n+1}} - \frac{(-q^2;q^2)_\infty}{(q)_\infty}\)$. 
What relationship do these series have with partial theta functions and mock theta functions?
\end{challenge*}
Andrews proves 
$$b(q) - d(q) = -c(q)$$
where $d(q):= \sum_{n=0}^\infty \frac{q^{n^2} (-q^2;q^2)_n}{(q)_{2n+1}} = \frac{(-q;q^2)_\infty}{(q;q^2)_\infty}$
The series $d(q^{-1})$ converges for $\abs{q}<1$. Moreover, 
$$d(q^{-1}) = \sum_{n=0}^\infty \frac{q^{2n+1} (-q^2;q^2)_n}{(q)_{2n+1}} = 0.$$

\subsection{Other renormalized series}

The  crank of an integer partition $\lambda$, denoted $c(\lambda)$,  is defined as 
$$c(\lambda):= \begin{cases} \lambda_1 & X_1(\lambda) =0 \\ \omega(\lambda) - X_1(\lambda) & X_1(\lambda) >0 \end{cases}$$
where $\lambda_1$ is the largest part of $\lambda$, $X_1(\lambda)$ is the number of parts of size $1$ in $\lambda$, 
and $\omega(\lambda)$ is the number of parts in $\lambda$ that are strictly larger than $X_1(\lambda)$. 
For $n >1$ let $M(m,n)$ be the number of partitions of $n$ with crank equal to $m$ and set 
$$M(m,n) = \begin{cases} -1 & (m,n) = (0,1) \\ 1 & (m,n) = (0,0), (1,1), (-1, 1) \\ 0  & n \le 1 \text{ otherwise} \end{cases}.$$
The crank generating function is $$C(z;q) = \sum_{n\ge 0} \sum_{m\in \Z} M(m,n) z^m q^n = \frac{(q)_\infty}{(zq)_\infty (z^{-1}q)_\infty}.$$

From this (see also Garvan \cite{garvanKrank}) one may deduce that 
$$\sum_{n=0}^\infty \sum_{m>0} M(m,n) q^n = \frac{1}{(q)_\infty} \sum_{n=0}^\infty (-1)^n q^{\frac{1}{2} n(n+1)}.$$
The series on the right is a partial theta function and has a well known asymptotic expansion radially toward every root 
of unity. 

Moreover, 
$$ cr(q):= \sum_{n=0}^\infty \frac{q^{n(n+1)}}{(q)_n^2} = \frac{1}{(q)_\infty} \sum_{n=0}^\infty (-1)^n q^{\frac{1}{2} n(n+1)}.$$
The series $cr(q^{-1}) = \sum_{n=0}^\infty \frac{1}{(q)_n^2}$ requires renormalization.  
The renormalized series 
$$\sum_{n=0}^\infty \( \frac{1}{(q)_n^2} - \frac{1}{(q)_\infty^2}\) = \frac{1}{(q)_\infty^2}  \( \sum_{n=1}^\infty \frac{q^n}{1-q^n} + \sum_{n=1}^\infty \frac{(-1)^{n-1} q^{\frac{1}{2} n(n+1)}}{1-q^n}\).$$
This series also has an interesting relationship with the crank generating function. 

The moments of the crank generating function have proven to be  combinatorially very interesting.   See the works of 
Andrews \cite{andrewsInvent, andrewsSpt, ack}, Garvan \cite{garvanSpt}, and Dixit and Ye \cite{dy}. 
Andrews, Chan and Kim \cite{ack} showed the first odd moment 
$$C_1(q):= \sum_{n\ge 0} \sum_{m\ge 0 } m M(m,n) q^n = \frac{1}{(q)_\infty} \sum_{n=1}^\infty \frac{(-1)^{n+1} q^{\frac{1}{2} n(n+1)}}{1-q^n}.$$

\vspace{.1in} 

Finally, we mention a series whose arithmetic is curious. 
Lovejoy \cite{lovejoyBailey} defined the series 
$$L_6(q):= \sum_{n=0}^\infty \frac{(q)_n^2 (-q)^n}{(q;q^2)_{n+1}} = \sum_{n=0}^\infty \sum_{r \le n} 
(-1)^{n+r} (2r+1) q^{n^2+ 2n - \frac{1}{2} r(r+1)}.$$
The series $L_6(q^{-1} ) = - \sum_{n=0}^\infty \frac{(q)_n^2}{(q;q^2)_{n+1}}$ requires renormalization and is 
given by 
$$\sum_{n=0}^\infty \( \frac{(q)_n^2}{(q;q^2)_{n+1}} - \frac{(q)_\infty^2}{(q;q^2)_\infty} \) = 
2q+ 5q^2 + 2q^3 + 2q^4 - 5q^5 - q^6 - 6q^7 - 2q^8 + 7q^9 - 4q^{10}+ 7q^{11} + \cdots $$
\begin{challenge*}
Find a $q$-hypergeometric expression for this renormalized sum. 
\end{challenge*}
\begin{challenge*}
Is there a modular object associated with these series?
\end{challenge*}

\vspace{.2in}
\subsection{A Final Challenge}
There remains a final challenge which is the most mysterious and
interesting.  Is there any structure to the series which are chosen as
$\cG[H](q)$?  In the work of Bringmann, Folsom, and the third author
on mock theta functions of weight $1/2$, the series $\cG[H](q)$ are
always equal to a modular form times a partial theta function.  In
this work we often have a modular form times a divisor function.  Can
the space of ghost terms be characterized?

\end{document}